\documentclass[a4paper,10pt]{article}
\usepackage[a4paper,top=3cm,bottom=2cm,left=4cm,right=4cm,marginparwidth=4cm]{geometry}
\usepackage{graphicx} 

\usepackage{subcaption}
\usepackage[T1]{fontenc}
\usepackage[utf8]{inputenc}
\DeclareUnicodeCharacter{0119}{\k{e}}  
\DeclareUnicodeCharacter{0105}{\k{a}}  
\DeclareUnicodeCharacter{017C}{\.{z}}  
\usepackage[normalem]{ulem}
\usepackage{amsmath}
\usepackage{amsthm}
\usepackage{amssymb}
\usepackage{thmtools}
\usepackage{thm-restate}
\usepackage{mathtools}
\usepackage{bm}
\usepackage{enumitem}
\usepackage{graphicx}
\usepackage{tikz}
\usetikzlibrary{shapes.geometric, positioning, backgrounds, arrows.meta, decorations.pathreplacing, calc}
\tikzset{source/.style={circle,fill=gray!70,draw,minimum size=0.5cm,inner 
sep=0pt}}
\tikzset{non-source/.style={circle,draw,minimum size=0.5cm,inner 
sep=0pt}}
\usepackage{framed}
\usepackage{tabularx}

\usepackage{todonotes}
\usepackage[numbers]{natbib}
\usepackage[ruled, linesnumbered]{algorithm2e}
\usepackage{hyperref}

\DeclareMathOperator{\diam}{diam}

\usepackage{cleveref}

\pgfmathsetmacro{\x}{1.447/sqrt(3)}
\theoremstyle{definition}

\newtheorem{thm}{Theorem}[section]
\newtheorem{lem}[thm]{Lemma}

\newtheorem{cor}[thm]{Corollary}

\newtheorem{prob}[thm]{Problem}

\newtheorem{rrule}{Reduction Rule}{\upshape\itshape}{\upshape\rmfamily}

\Crefname{section}{Section}{Sections}
\Crefname{figure}{Figure}{Figures}
\Crefname{lem}{Lemma}{Lemmas}
\Crefname{thm}{Theorem}{Theorems}
\Crefname{cor}{Corollary}{Corollaries}
\Crefname{obs}{Observation}{Observations}
\Crefname{prop}{property}{properties}
\Crefname{rrule}{Rule}{Rules}

\newcommand{\problemdef}[3]{%
	\setlength{\FrameSep}{3pt}
	\begin{framed}
		\normalsize\textsc{#1} \smallskip \\
		\begin{tabularx}{0.95\textwidth}{@{}l@{\hspace{3pt}}X}
			\normalsize\textbf{Input:}    & \normalsize#2 \\
			\normalsize\textbf{Question:} & \normalsize#3
		\end{tabularx}
	\end{framed}
}
\newcommand{\PROB}[1]{{{\normalfont\textsc{#1}}}\xspace}

\newcommand{\ColorLONG}{\PROB{$k$-Coloring}}

\newcommand{\TC}{\PROB{3-Coloring}}
\newcommand{\LCLONG}{\PROB{List $k$-Coloring}}

\newcommand{\LTC}{\PROB{List $3$-Coloring}}
\newcommand{\LTCD}{\PROB{List $3$-Coloring Diameter-$d$}}
\newcommand{\LTCDth}{\PROB{List $3$-Coloring Diameter-$3$}}
\newcommand{\LTCDtw}{\PROB{List $3$-Coloring Diameter-$2$}}
\newcommand{\LTCDCF}{\PROB{List $3$-Coloring $C_4$-free Diameter-$2$}}

\title{List $3$-coloring $C_4$-free graphs of diameter-$2$ in polynomial-time}
\author{Yukihiro 
Murakami\thanks{\href{mailto:y.murakami@tudelft.nl}{y.murakami@tudelft.nl}}}
\date{%
	Delft Institute of Applied Mathematics, Delft University of Technology, the 
	Netherlands\\[2ex]%
	\today}

\begin{document}

\maketitle

\begin{abstract}
We show that list $3$-coloring a~$C_4$-free graph of diameter-$2$ can be done in polynomial-time.
Our algorithm is based on a structural characterization showing that many such graphs are not~$3$-colorable. 
In particular, we show that~$C_4$-free graphs of diameter-$2$ without universal vertices, where the maximum degree is at least~$17$, are not~$3$-colorable.
\end{abstract}

\section{Introduction}
 
For an integer~$k\ge 1$, we write~$[k] = \{1,2,\ldots, k\}$. 
Given is a finite simple unweighted connected graph $G=(V,E)$.
A map~$c:V\rightarrow [k]$ for an integer~$k$ is called a~\emph{$k$-coloring} if~$c(u)\ne c(v)$ for all edges~$uv\in E$.
In such cases, we say that~$G$ is~\emph{$k$-colorable}.
The smallest~$k$ for which~$G$ admits a~$k$-coloring is called the \emph{chromatic number} of~$G$.

\problemdef{\ColorLONG}{
	A graph~$G=(V,E)$.
}{
	Is $G$ $k$-colorable?
}

One can also consider a more general variant called \emph{list $k$-coloring}.
Each vertex~$v$ is given a list~$L(v)\subseteq [k]$ of colors it may have; if there exists a~$k$-coloring~$c:V\rightarrow [k]$ where~$c(v)\in L(v)$, then we say that the graph is~\emph{list $k$-colorable}.
Note that~list $k$-colorable graphs are necessarily~$k$-colorable, but not vice versa. 

\problemdef{\LCLONG}{
	A graph~$G=(V,E)$ and a list~$L:V\rightarrow 2^{[k]}$.
}{
	Is~$G$ list~$k$-colorable?
}

\textsc{2-Coloring} is equivalent to showing that a graph is bipartite. $\ColorLONG$, and therefore $\LCLONG$, on the other hand, is NP-complete for all~$k\ge3$~\cite{Lovasz1973}.
Since, much attention has been devoted to solving the problems on restricted graphs, such as those with forbidden (induced) subgraphs~\cite{kral2001complexity,chudnovsky2024four}, those on bounded diameter \cite{mertzios2016algorithms,martin2019colouring}, or a combination thereof~\cite{martin2022colouring,klimovsova20233}.
There are numerous thorough surveys on the topic \cite{tuza1997graph,randerath2004vertex,golovach2017survey}.
We study list $3$-coloring on bounded diameter.

\paragraph{Bounded diameter.}

\problemdef{\LTCD}{
	A graph~$G=(V,E)$ of diameter-$d$ and a list~$L:V\rightarrow 2^{[3]}$.
}{
	Is~$G$ list~$3$-colorable?
}

Mertzios and Spirakis showed that $\TC$ remains NP-complete when the input is restricted to graphs of diameter-$3$ (and even triangle-free)~\cite{mertzios2016algorithms}.
In the same paper, they showed that under the Exponential Time Hypothesis, the problem cannot be solved in~$2^{o(\sqrt{n})}$ time; they gave a subexponential-time algorithm to solve $\LTCDth$~\cite{mertzios2016algorithms}.
This was subsequently improved by D\k{e}bski, Piecyk, and Rz\k{a}\.{z}ewski to run in $2^{O((n\log n)^{2/3})}$ time~\cite{debski2022faster} and later by Groenland, Koerts, and Spirkl to $2^{O(n^{2/3-\epsilon})}$ for any~$\epsilon<1/33$~\cite{groenland2026faster}.

For graphs of diameter-$2$, the computational complexity of both~$\TC$ and~$\LTC$ remains open.
Mertzios and Spirakis gave a~$2^{O(\sqrt{n\log n})}$ time algorithm based on dominating sets~\cite{mertzios2016algorithms}. The fastest algorithm thus far runs in $2^{O(n^{1/3}\log^2n)}$ time due to D\k{e}bski, Piecyk, and Rz\k{a}\.{z}ewski~\cite{debski2022faster}.

To better understand the boundary between the intractable (potentially NP-hard) and the tractable (polynomial-time solvable), many have investigated \LTCDtw in combination with forbidden subgraphs.
One of particular interest is in considering \emph{induced $(H_1,H_2,\ldots, H_p)$-free} graphs, which are graphs that do not contain an induced subgraph isomorphic to~$H_i$, for all~$i\in [p]$.
A similar notion is an~\emph{$H$-free} graph, which does not contain a subgraph isomorphic to~$H$.
It has been shown that~$\LTCDtw$ is polynomial-time solvable for
\begin{itemize}
    \item graphs in which \emph{centers of induced claws}\footnote{A claw is the graph~$K_{1,3}$. A center of a claw is the vertex with degree-$3$. A graph is quasi-claw-free if for any two vertices~$x,y$ which are distance-$2$ apart, there exists a common neighbour~$z$ such that all neighbours of~$z$ (excluding~$x,y$) are neighbours of~$x$ or~$y$.}  form an independent set and \emph{quasi-claw-free} graphs~\cite{martin2022colouringGirth}.
    \item induced $C_s$-free graphs for~$s=5,6$~\cite{martin2022colouring}.
    \item induced $(C_4, C_s)$-free graphs for~$s\ge 3$ (shown for~$s\in\{3,5,6,7,8,9\}$ in \cite{martin2022colouring}, generalized to other~$s$ values in~\cite{klimovsova20233}).
    \item induced $(C_3,C_7)$-free graphs \cite{klimovsova20233}.
\end{itemize} 
Most of these results follow a similar recipe.
Color a subset of the vertices and propagate the color constraints, i.e. for uncolored vertices, remove all colors which appear on their neighbours from the list.
Check that the remaining vertices must each have a list of size at most~$2$; as \textsc{2-List-Coloring}\footnote{Note that this is different from \textsc{List $2$-Coloring}. Here, lists do not have to be subsets of~$\{1,2\}$, but they are subsets of size $2$.} can be solved in polynomial-time~\cite{edwards1986complexity}, it remains to show that finding the first suitable pre-coloring takes polynomial-time.

\paragraph{Our contribution.}
 
We take a structural approach and show that most $C_4$-free graphs of diameter-$2$ are not~$3$-colorable.

\problemdef{\LTCDCF}{
	A $C_4$-free graph~$G=(V,E)$ of diameter-$2$ and a list~$L$.
}{
	Is~$G$ list~$3$-colorable?
}

\begin{restatable}[]{thm}{LTCDtwfour}\label{thm:main}
	\LTCDCF is polynomial-time solvable.
\end{restatable}

To show \Cref{thm:main}, we show that most~$C_4$-free graphs of diameter-$2$ are not~$3$-colorable if the maximum degree is at least~$17$.
Determining the maximum degree of a graph can be done in polynomial-time.
The list $3$-colorability of the remaining finitely many graphs can be hardcoded into a lookup table.  

We give a quick overview of the proof.
We show that the $C_4$-free and diameter-$2$ conditions impose a rigid structure on the graph.
We start by fixing a vertex~$v_0$ of maximum degree~$\Delta$.
Letting~$N_i$ denote vertices which are distance-$i$ away from~$v_0$, assuming~$N_2\ne \emptyset$, we show that vertices in~$N_2$ can be partitioned based on their unique adjacent vertex in~$N_1$.
There are~$\Delta$ parts in the partition; we show that each part contains the same number of elements, which is either~$\Delta-2$ or~$\Delta-1$.
We also show that every vertex must be of  degree~$\Delta-1$ or~$\Delta$, and that between every pair of parts in~$N_2$ which are `adjacent', the vertices therein form a perfect matching.
We show that every vertex in~$N_2$ is of one of three \emph{types}.
We then exhaustively check that most graphs containing all possible Types of~$N_2$ vertices are not~$3$-colorable.
This is done by showing that there are too many triangles which force an unachievable number of vertices of a certain color.

While this is useful in showing a result in the  \textsc{Coloring} context, we envision that the improved structural understanding of the class of~$C_4$-free graphs of diameter-$2$ can be of use in other related areas.

\paragraph{Structure.} In \Cref{sec:Preliminaries}, we detail necessary definitions and preliminaries.
In \Cref{sec:Characterization}, we give structural results for $C_4$-free graphs of diameter-$2$.
In \Cref{sec:Color}, we prove necessary conditions for a proper~$3$-coloring to exist for $C_4$-free graphs of diameter-$2$.
In \Cref{sec:ColorTypes}, we show that most graphs in the class are not~$3$-colorable.
In \Cref{sec:Proof}, we prove \Cref{thm:main}.
We give concluding remarks in \Cref{sec:Discussion}.

\section{Preliminaries}\label{sec:Preliminaries}

Let~$G=(V,E)$ be a graph.
Let~$u,v\in V$.
We say that~$u$ \emph{covers~$v$ via adjacency} if~$uv\in E$.
We say that a vertex~$u$ \emph{covers~$v$ via a proxy~$w$} (or just \emph{covers~$v$ via proxy} if the context is clear) if there exists a vertex~$w$ such that~$uw,wv\in E$.
The \emph{degree} of~$u$, denoted~$\deg_G(u)$, is the number of edges incident with~$u$.
We remove the subscript and write~$\deg(u)$ if the reference graph is clear.
The \emph{minimum degree} of~$G$ is~$\delta(G) = \min_{v} \deg(v)$; the \emph{maximum degree} of~$G$ is~$\Delta(G) = \max_{v}\deg(v)$.

We consider unweighted graphs.
The \emph{length} of a path is the number of edges it contains.
The \emph{distance} between two vertices is the number of edges in a shortest path between them.
The \emph{diameter} of~$G$, written~$\diam(G)$ is the maximum distance between any pair of its vertices.
Note that in a graph of diameter-$2$, $u$ covers~$v$ via adjacency or via proxy, or both.

We write~$N^G_i(v) = \{u: d(u,v) = i\}$ to denote the set of vertices which are distance~$i$ away from~$v$.
We drop the superscript if there is no ambiguity in the reference graph.
Let~$u,v\in V$.
We \emph{merge}~$u$ and~$v$ by replacing them with a new vertex~$w$, where~$N_1(w) = N_1(u)\cup N_1(v)$.

Let~$S\subseteq V$. 
The \emph{induced subgraph} $G[S]$ is the graph whose vertex set is~$S$ and whose edge set consists of all elements in~$E$ with both endpoints in~$S$.
We say that~$G$ is~\emph{induced $(H_1, H_2,\ldots, H_p)$-free} if~$G$ does not contain an induced subgraph isomorphic to~$H_i$ for all~$i\in[p]$.
We say that~$G$ is \emph{$H$-free} if~$G$ does not contain a subgraph that is isomorphic to~$H$.

In this paper, we consider~$C_4$-free graphs. 
These are induced $(C_4,D_4,K_4)$-free graphs, where~$K_4$ is the \emph{complete graph} on~$4$ vertices and~$D_4$ is a \emph{diamond graph}, a graph on four vertices and five edges (so a~$K_4$ without an edge).
Since graphs containing $K_4$ as a subgraph are not~$3$-colorable, and since identifying a~$K_4$ takes polynomial-time, one can also restrict to induced~$(C_4,D_4)$-graphs.
Nevertheless, for ease of notation, we formulate our results for~$C_4$-free graphs.
Note that there is a so-called \emph{diamond elimination} rule  \cite{mertzios2016algorithms}.
It is not immediately clear if this rule preserves induced~$C_4$-freeness. 
Thus, we have decided to also exclude~$D_4$ as an induced subgraph.
See \Cref{sec:Discussion} for a discussion on the matter.



The following lemma is in the spirit of Observation 1 in \cite{mertzios2016algorithms}.

\begin{lem}\label{lem:Irreducible}
    Let~$G=(V,E)$ be a~$C_4$-free graph.
    If~$G$ has more than two vertices, then for every~$u\in V$, the maximum degree of~$G[N_1(u)]$ is~$1$.
\end{lem}
\begin{proof}
    If there were a vertex~$y$ in $G[N_1(u)]$ of degree-$2$ with neighbours~$x$ and~$z$, then~$uxyz$ is a~$C_4$, which is prohibited.
\end{proof}

\section{Characterizing \texorpdfstring{$C_4$}{C4}-free graphs of diameter-\texorpdfstring{$2$}{2}}\label{sec:Characterization}

We introduce notation which will be used in the rest of this paper.
We write~$G=(V,E)$ to mean a $C_4$-free graph of diameter-$2$.
We let~$\Delta$ denote the maximum degree of~$G$, and let~$v_0$ be a vertex with degree~$\Delta$.
We call~$v_0$ a \emph{root}.
If~$\Delta=2$, then~$G$ must be isomorphic to~$C_5$, which is~$3$-colorable. 
We assume~$\Delta\ge 3$.
Let~$N_1 = N_1(v) = \{x_1,x_2,\ldots, x_\Delta\}$ be the set of neighbours of~$v_0$, and let~$N_2 = N_2(v)$.

A vertex that is adjacent to every other vertex is called a \emph{universal vertex}.
\LTCDCF is polynomial-time solvable if there exists a universal vertex.
\begin{lem}\label{lem:N2Empty}
    If~$N_2=\emptyset$, then \LTCDCF is polynomial-time solvable.
\end{lem}
\begin{proof}
    If there exists a proper list $3$-coloring, there exists one where~$v_0$ is given one of at most three colors in~$L(v_0)$.
    Each guess returns an instance of list-$2$ coloring, which can be solved in polynomial-time~\cite{edwards1986complexity}.
\end{proof}
So we assume henceforth that~$N_2$ is non-empty.
See \Cref{fig:PotentialYESinstances} for all~$C_4$-free graphs of diameter-$2$ when~$\Delta\le 3$.

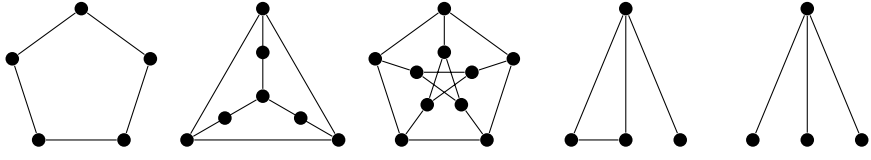
\begin{figure}
    \centering
    \begin{tikzpicture}[
    every node/.style={circle, fill=black, inner sep=1.8pt},
    scale=1.2, 
    >=stealth
]


    \begin{scope}[shift={(0, 0.647)}]
        \foreach \i in {1,...,5} {
            \node (g1-\i) at ({90 - (\i-1)*72}:0.8) {};
        }
        \draw (g1-1) -- (g1-2) -- (g1-3) -- (g1-4) -- (g1-5) -- (g1-1);
    \end{scope}

    \begin{scope}[shift={(2, 0.447)}]
        \node (g2-top) at (0,1.0) {};
        \node (g2-center) at (0,0.0353) {};
        \node (g2-sub-top) at (0,0.51765) {};
        \node (g2-sub-left) at (-0.4178,-0.20585) {};
        \node (g2-sub-right) at (0.4178,-0.20585) {};
        \node (g2-left-bot) at (-\x,-0.447) {};
        \node (g2-right-bot) at (\x,-0.447) {};
        
        
        \draw (g2-top) -- (g2-left-bot) -- (g2-right-bot) -- (g2-top);
        \draw (g2-top) -- (g2-center);
        \draw (g2-left-bot) -- (g2-center);
        \draw (g2-right-bot) -- (g2-center);
    \end{scope}

    \begin{scope}[shift={(4, 0.647)}]
        \foreach \i in {1,...,5} {
            \node (out-\i) at ({90 - (\i-1)*72}:0.8) {};
        }
        \draw (out-1) -- (out-2) -- (out-3) -- (out-4) -- (out-5) -- (out-1);
        
        \foreach \i in {1,...,5} {
            \node (in-\i) at ({90 - (\i-1)*72}:0.32) {};
        }
        \draw (in-1) -- (in-3) -- (in-5) -- (in-2) -- (in-4) -- (in-1);
        
        \foreach \i in {1,...,5} {
            \draw (out-\i) -- (in-\i);
        }
    \end{scope}

    \begin{scope}[shift={(6, 0.447)}]
        \node (g2-top) at (0,1.0) {};
        \node (g2-sub-left) at (-0.6,-0.447) {};
        \node (g2-sub-right) at (0.6,-0.447) {};
        \node (g2-bot) at (0,-0.447) {};
        
        
        \draw (g2-top) -- (g2-bot);
        \draw (g2-top) -- (g2-sub-left);
        \draw (g2-top) -- (g2-sub-right);
        \draw (g2-bot) -- (g2-sub-left);
    \end{scope}

    \begin{scope}[shift={(8, 0.447)}]
        \node (g2-top) at (0,1.0) {};
        \node (g2-sub-left) at (-0.6,-0.447) {};
        \node (g2-sub-right) at (0.6,-0.447) {};
        \node (g2-bot) at (0,-0.447) {};
        
        
        \draw (g2-top) -- (g2-bot);
        \draw (g2-top) -- (g2-sub-left);
        \draw (g2-top) -- (g2-sub-right);
    \end{scope}

\end{tikzpicture}
    \caption{All $C_4$-free graphs of diameter-$2$ when~$\Delta\le 3$.
    Note that they are $3$-colorable, which means they are potentially list~$3$-colorable.
    The three graphs on the left do not have universal vertices; the two graphs on the right do have universal vertices.}
    \label{fig:PotentialYESinstances}
\end{figure}




\subsection{Overview of results}
We give an overview of results in \Cref{sec:Characterization}. 
By taking a vertex~$v_0$ with $\deg(v_0) = \Delta$ as root, we show that~$N_2= N_2(v_0)$ has a canonical partition~$N_2 = \bigcup_{i=1}^\Delta V_i,$ where~$V_i\cap V_j = \emptyset$ for all~$i\ne j$,~$|V_1| = |V_2| = \cdots = |V_\Delta|$ (\Cref{lem:V_iUniformSize}), and $|V_i|\in \{\Delta-2,\Delta -1\}$ (\Cref{lem:V_iSizeLowerBound}).
Every vertex in~$N_2$ is of Types-$1$, $2$, or~$3$ (see \Cref{subsec:Type}), and these vertices are of degrees~$\Delta,\Delta-1$, and~$\Delta$, respectively (\Cref{lem:T1Delta,lem:T2Delta-1,lem:T3Delta}). 
If~$N_2$ contains a Type-$2$ or a Type-$3$ vertex, then~$|V_i| = \Delta-2$; if~$|V_i| = \Delta-1$, then~$N_2$ contains only Type-$1$ vertices (\Cref{lem:adj=Delta-2}).
If~$N_2$ contains a Type-$2$ vertex, then a different rooting gives a canonical partition that contains a Type-$1$ vertex (\Cref{lem:RootShift}).
This reduces the number of cases we need to consider in \Cref{sec:ColorTypes}.
Lastly, there can only be~$C_3$'s in~$G[N_2]$ between two adjacent vertices which are both not Type-$2$ (\Cref{lem:TypeTriangles}).
In general, this means there are usually a lot of~$C_3's$ in~$G[N_2]$. 
We use this fact to show non-$3$-colorability results in \Cref{sec:Color,sec:ColorTypes}.

\subsection{A canonical partition of \texorpdfstring{$N_2$}{N2}}\label{subsec:N2Study}

\begin{lem}\label{lem:N2SingleEdge}
    Every vertex in~$N_2$ has exactly one neighbour in~$N_1$.
\end{lem}
\begin{proof}
    Let~$u\in N_2$.
    The vertex~$u$ must cover~$v_0$ via proxy, so it must have at least one neighbour in~$N_1$.
    If there are two such neighbours, say~$x_i,x_j\in N_1$, then~$v_0x_iux_j$ forms a~$C_4$.
\end{proof}

\begin{lem}\label{lem:N1Neighbour}
    Every vertex in~$N_1$ has at least one neighbour in~$N_2$.
\end{lem}
\begin{proof}
    Let~$x\in N_1$. We wish to show that~$x$ has a neighbour in~$N_2$.
    Let~$u\in N_2$. 
    By \Cref{lem:N2SingleEdge},~$uy\in E$ where~$y\in N_1$.
    If $xy\notin E$, then~$u$ must cover~$x$ via a proxy~$w\in N_2$, i.e.,~$uw,wx\in E$.
    So~$x$ has a neighbour in~$N_2$.

    So suppose~$xy\in E$.
    Since~$\Delta\ge 3$, $v_0$ has a neighbour~$z$ that is not~$x$ nor~$y$.
    By~\Cref{lem:Irreducible},~$G[N_1]$ does not contain a~$P_3$, and so~$zx,zy\notin E$.
    Then~$u$ must cover~$z$ via a proxy~$v\in N_2$, and~$v$ must cover~$x$ via a proxy~$w\in N_2$.
    So~$x$ has a neighbour in~$N_2$.
\end{proof}


By \Cref{lem:N2SingleEdge}, every vertex in~$N_2$ is adjacent to exactly one vertex in~$N_1$.
We partition the vertex set~$N_2$ based on this adjacency; in particular, let~$V_i$ denote the set of vertices in~$N_2$ which are adjacent to~$x_i$, for~$i\in[\Delta]$.
Note that~$|V_i|\ge 1$ for all $i$, by \Cref{lem:N1Neighbour}.
We call such a partition~$V_1, V_2,\ldots, V_\Delta$ a \emph{canonical partition} (with respect to~$v_0$) of~$N_2$.
See \Cref{fig:Example} for an example.


We say that two parts~$V_i, V_j$ are \emph{adjacent} if a vertex of~$V_i$ has a neighbour in~$V_j$.

\begin{lem}\label{lem:x_ix_jNonAdjacent}
    Suppose $x_i$ and~$x_j$ are non-adjacent. Then~$V_i$ and~$V_j$ are adjacent, and the edges between $V_i$ and~$V_j$ form a perfect matching.
    In particular,~$|V_i|=|V_j|$.
\end{lem}
\begin{proof}
    As~$x_ix_j\notin E$, every vertex in~$V_i$ must cover~$x_j$ via a proxy in~$V_j$. 
    Similarly, every vertex in~$V_j$ must cover~$x_i$ via a proxy in~$V_i$.
    Two vertices in~$V_i$ cannot have a common neighbour apart from~$x_i$, as that would result in a~$C_4$.
    So~$|V_i|\le |V_j|$. Similarly,~$|V_j|\le |V_i|$.
    We obtain our perfect matching and the claim follows. 
\end{proof}

Note also that if~$V_i$ and~$V_j$ are non-adjacent, then~$x_i$ and~$x_j$ must be adjacent.


\begin{lem}\label{lem:V_iAdjacency}
    $V_i$ is adjacent to at least~$\Delta-2$ parts for all~$i\in [\Delta]$.
\end{lem}
\begin{proof}
    By \Cref{lem:Irreducible},~$x_i$ has at most one neighbour in~$N_1$, and so~$x_i$ is non-adjacent to at least~$\Delta-2$ vertices in~$N_1$.
    By \Cref{lem:x_ix_jNonAdjacent},~$V_i$ is adjacent to at least~$\Delta-2$ parts. 
\end{proof}

\begin{lem}\label{lem:V_iUniformSize}
    The canonical partition is balanced, i.e., $|V_1|=|V_2|=\cdots = |V_\Delta|$.
\end{lem}
\begin{proof}
    Let~$x_i\in N_1$.
    If~$x_i$ is not adjacent to any vertices in~$N_1$, then we are done by~\Cref{lem:x_ix_jNonAdjacent}.
    So suppose~$x_ix_j\in E$ for some~$x_j\in N_1$. 
    By~\Cref{lem:Irreducible},~$x_i$ and~$x_j$ have no other neighbours in~$N_1$.
    By invoking~\Cref{lem:x_ix_jNonAdjacent} twice, we obtain~$|V_i|=|V_k|=|V_j|$ for each~$k\in [\Delta]\setminus\{i,j\}$.
    Such a set~$[\Delta]\setminus\{i,j\}$ is non-empty as~$\Delta\ge 3$, which gives the required result.
\end{proof}

\begin{lem}\label{lem:V_iSizeLowerBound}
    $|V_i|\in\{\Delta-2, \Delta-1\}$ for all~$i \in [\Delta]$.
\end{lem}
\begin{proof}
    For the upper bound, observe that~$\Delta \ge\deg(x_i)\ge |V_i|+1$, where the $+1$ term comes from the edge~$x_iv_0$. 
    We now show the lower bound (see \Cref{fig:V_iSizeLowerBound} for an illustration).
    If there are two non-adjacent parts, select these as~$V_i$ and~$V_j$.
    Otherwise, select any pair of parts.
    Let~$v\in V_i$.
    We study how vertices in~$V_j$ are covered by~$v$ via proxies.
    By \Cref{lem:x_ix_jNonAdjacent} and \Cref{lem:V_iAdjacency},~$v$ is adjacent to a vertex in every part that is not~$V_j$, so~$v$ is adjacent to a vertex in at least~$\Delta-2$ parts.
    If~$|V_j|\le \Delta-3$, then by the pigeonhole principle, two neighbours of~$v$ must have the same neighbour in~$V_j$. 
    This gives a~$C_4$, and thus~$|V_j|\ge \Delta-2$.
    By~\Cref{lem:V_iUniformSize},~$|V_i|=|V_j|\ge \Delta-2$ for all~$i\in [\Delta]$.
\end{proof}

\begin{figure}
    \centering
    \begin{tikzpicture}[
    scale=1.0,
    vertex/.style={circle, fill=black, inner sep=2pt},
    every label/.style={font=\small},
    setlabel/.style={font=\large}
]

    
    \draw[thick] (1,2.8) ellipse (1 and 0.5);
    \node[setlabel] at (0, 3.4) {$V_i$};
    \node[vertex] (u) at (1, 2.8) {};
    \node[anchor=south] at (u.north) {$u$};

    \draw[thick] (3.8, 2.8) ellipse (1 and 0.5);
    \node[setlabel] at (4.8, 3.4) {$V_j$};
    \node[vertex] (vj1) at (3.2, 2.8) {}; 
    \node[vertex] (vj2) at (3.8, 2.8) {}; 
    \node[vertex] (vj3) at (4.4, 2.8) {}; 
    \node[font=\tiny] at (4.1, 2.8) {$\dots$};

    
    \coordinate (c1) at (0, 0);

    \draw[thick] (c1) circle (0.7);
    \node[setlabel] at ($(c1) + (-1, -0.4)$) {$V_1$};
    \node[vertex] (v1_node) at (c1) {};

    \coordinate (c2) at (2, 0);
    \draw[thick] (c2) circle (0.7);
    \node[vertex] (v2_node) at (c2) {};
    \node[setlabel] at ($(c2) + (1,-0.4)$) {$V_2$};

    \node[vertex] (vdel) at (4.8, 0) {};
    \draw[thick] (vdel) ellipse (0.7);
    \node[setlabel] at ($(vdel) + (1, -0.4)$) {$V_\Delta$};

    \node[font=\LARGE] at (3.45, 0) {$\dots$};


    \draw[ultra thick, dotted] (2.4, 3.4) -- (2.4, 2.2);

    \draw[thick] (vj3) -- (vdel);           
    \draw[thick] (u) -- (v1_node);           
    \draw[thick] (v1_node) -- (vj1);         

    \draw[thick] (u) -- (v2_node); 
    \draw[thick] (vj2) -- (v2_node);                                    

    \draw[thick] (u) -- (vdel);

    \draw [decorate,decoration={brace,amplitude=12pt,mirror},thick]
    (-1, -1) .. controls (3.3, -1) and (3.5, -1) .. (5.8,-1)
    node [black,midway,xshift=-0.8cm,yshift=-0.7cm] {$\Delta - 2$};

\end{tikzpicture}
    \caption{A visualization of the lower bound~$|V_i|\ge\Delta-2$ in the proof of \Cref{lem:V_iSizeLowerBound}, for the case when there are non-adjacent parts. 
    The figure shows a subgraph of~$G[N_2]$, illustrating how~$u\in V_i$ covers the vertices of a non-adjacent part~$V_j$. 
    The dotted line between~$V_i$ and~$V_j$ indicate the non-adjacency.
    The additional vertices in~$V_1,\ldots, V_{j-1},V_{j+1}, V_\Delta$ are omitted for clarity.}
    \label{fig:V_iSizeLowerBound}
\end{figure}
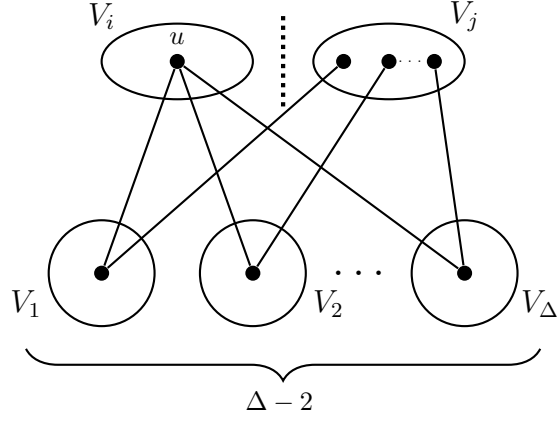

\begin{lem}\label{lem:adj=Delta-2}
    If there are non-adjacent parts, then~$|V_i|=\Delta-2$ for all~$i\in [\Delta]$.
\end{lem}
\begin{proof}
    Let~$V_i,V_j$ be non-adjacent parts.
    Let~$u_i\in V_i$.
    Then~$u_i$ covers the vertices of~$V_j$ only via proxy in other parts; there are~$\Delta-2$ such edges, and thus~$|V_j|\le \Delta-2$.
    Consequently, by \Cref{lem:V_iUniformSize,lem:V_iSizeLowerBound}, we obtain~$|V_i| = \Delta-2$ for all~$i\in [\Delta]$.
\end{proof}
Note that if~$|V_i| = \Delta-1$, then all parts are pairwise adjacent, i.e.,~$N_1$ forms an independent set.

\subsection{Classifying vertices in \texorpdfstring{$N_2$}{N2}}\label{subsec:Type}

Let~$i\in [\Delta]$ and let~$v\in V_i$. We say that~$v$ is of
\begin{itemize}
    \item \emph{Type-$1$} if~$x_i$ has no neighbours in~$N_1$ and~$v$ has no neighbours in~$V_i$;
    \item \emph{Type-$2$} if~$x_i$ has a neighbour in~$N_1$ and~$v$ has no neighbours in~$V_i$;
    \item \emph{Type-$3$} if~$x_i$ has a neighbour in~$N_1$ and~$v$ has one neighbour in~$V_i$;
\end{itemize}

\begin{lem}\label{lem:1of3Types}
    Every vertex in~$N_2$ is of Type-$1$, of Type-$2$, or of Type-$3$.
\end{lem}
\begin{proof}
    Since~$N_1$ is~$P_3$-free by \Cref{lem:Irreducible},~$x_i$ has at most one neighbour in~$N_1$.
    Suppose first that~$x_i$ has a neighbour in~$N_1$.
    Let~$v\in V_i$.
    Since $V_i\subseteq N_1(x_i)$,~$v$ has at most one neighbour in~$V_i$ by \Cref{lem:Irreducible}.
    Then~$v$ is either of Type-$2$ or of Type-$3$.
    
    So suppose~$x_i$ has no neighbours in~$N_1$, and suppose~$v$ has a neighbour~$y$ in~$V_i$.
    Note that~$v$ must cover~$x_j$ via proxy in $V_j$, for all~$j\in[\Delta]\setminus\{i\}$.
    This gives~$\deg(v)\ge\Delta-1 + 2 = \Delta+1$, where the $+2$ term comes from the edges~$vy, vx_i$.
    This contradicts our choice of~$\Delta$, and so~$v$ cannot have a neighbour in~$V_i$.
    Thus~$v$ must be of Type-$1$.
\end{proof}

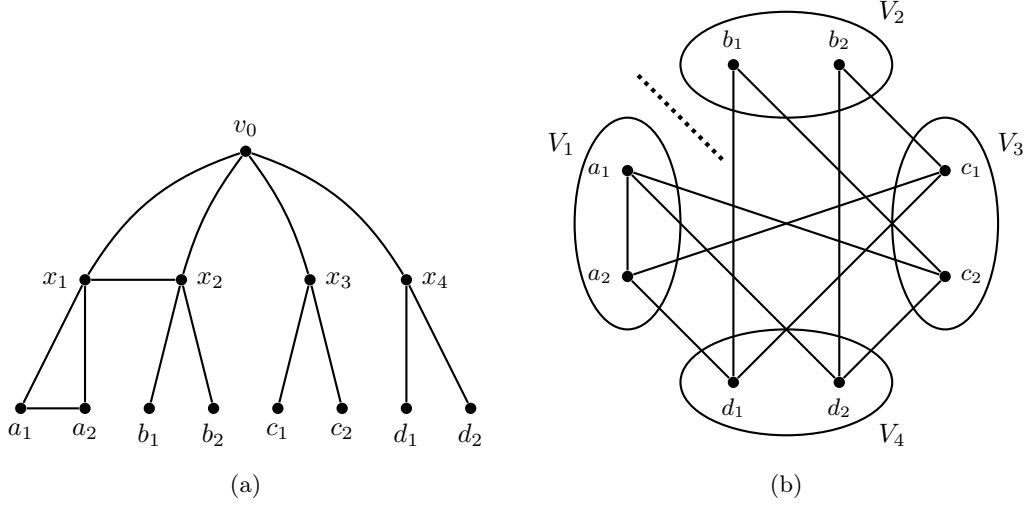
\begin{figure}
\centering
    
    \begin{subfigure}[b]{0.45\textwidth}
        \centering
        \makebox[\linewidth][c]{%
        \begin{tikzpicture}[
            scale=0.85,
            vertex/.style={circle, fill=black, inner sep=1.5pt}
        ]
            \node[vertex] (v0) at (3.5, 4) [label=above:$v_0$] {};
            
            \node[vertex] (x1) at (1, 2) [label=left:$x_1$] {};
            \node[vertex] (x2) at (2.5, 2) [label=right:$x_2$] {};
            \node[vertex] (x3) at (4.5, 2) [label=right:$x_3$] {};
            \node[vertex] (x4) at (6, 2) [label=right:$x_4$] {};
            
            \node[vertex] (a1) at (0, 0) [label=below:$a_1$] {};
            \node[vertex] (a2) at (1, 0) [label=below:$a_2$] {};
            \node[vertex] (b1) at (2, 0) [label=below:$b_1$] {};
            \node[vertex] (b2) at (3, 0) [label=below:$b_2$] {};
            \node[vertex] (c1) at (4, 0) [label=below:$c_1$] {};
            \node[vertex] (c2) at (5, 0) [label=below:$c_2$] {};
            \node[vertex] (d1) at (6, 0) [label=below:$d_1$] {};
            \node[vertex] (d2) at (7, 0) [label=below:$d_2$] {};
            
            \draw[thick] (v0) to[bend right=20] (x1);
            \draw[thick] (v0) to[bend right=10] (x2);
            \draw[thick] (v0) to[bend left=10] (x3);
            \draw[thick] (v0) to[bend left=20] (x4);
            
            \draw[thick] (x1) -- (x2);
            
            \draw[thick] (x1) -- (a1);
            \draw[thick] (x1) -- (a2);
            \draw[thick] (x2) -- (b1);
            \draw[thick] (x2) -- (b2);
            \draw[thick] (x3) -- (c1);
            \draw[thick] (x3) -- (c2);
            \draw[thick] (x4) -- (d1);
            \draw[thick] (x4) -- (d2);
            
            \draw[thick] (a1) -- (a2);
        \end{tikzpicture}}
        \caption{} 
        \label{fig:graph_I}
    \end{subfigure}
    \hfill
    \begin{subfigure}[b]{0.45\textwidth}
        \centering
        \makebox[\linewidth][c]{%
        \begin{tikzpicture}[
            scale=0.7,
            vertex/.style={circle, fill=black, inner sep=1.5pt},
            every label/.style={font=\small}
        ]
            \node[vertex] (A1) at (0, 4) [label=left:$a_1$] {};
            \node[vertex] (A2) at (0, 2) [label=left:$a_2$] {};
            
            \node[vertex] (B1) at (2, 6) [label=above:$b_1$] {};
            \node[vertex] (B2) at (4, 6) [label=above:$b_2$] {};
            
            \node[vertex] (C1) at (6, 4) [label=right:$c_1$] {};
            \node[vertex] (C2) at (6, 2) [label=right:$c_2$] {};
            
            \node[vertex] (D1) at (2, 0) [label=below:$d_1$] {};
            \node[vertex] (D2) at (4, 0) [label=below:$d_2$] {};
            
            \draw[thick] (0, 3) ellipse (1 and 2);
            \draw[thick] (3, 6) ellipse (2 and 1);
            \draw[thick] (6, 3) ellipse (1 and 2);
            \draw[thick] (3, 0) ellipse (2 and 1);
            
            \node at (-1.25, 4.5) {$V_1$};
            \node at (5, 7) {$V_2$};
            \node at (7.25, 4.5) {$V_3$};
            \node at (5, -1) {$V_4$};
            
            \draw[thick] (A1) -- (A2);
            
            \draw[thick] (B1) -- (C2);
            \draw[thick] (B2) -- (C1);
            \draw[thick] (B1) -- (D1);
            \draw[thick] (B2) -- (D2);
            \draw[thick] (A1) -- (D2);
            \draw[thick] (A2) -- (D1);
            \draw[thick] (A1) -- (C2);
            \draw[thick] (A2) -- (C1);
            \draw[thick] (D1) -- (C1);
            \draw[thick] (D2) -- (C2);
            

            \draw[ultra thick, dotted] (0.2,5.8) -- (1.8,4.2);
            
        \end{tikzpicture}}
        \caption{} 
        \label{fig:graph_II}
    \end{subfigure}
    
    \caption{A~$C_4$-free~$G$ of diameter-$2$ with maximum degree~$4$. (a) $G$ without the edges in~$E(G[N_2])$. $c_1,c_2,d_1,d_2$ are Type-$1$ vertices, $b_1,b_2$ are Type-$2$ vertices, and~$a_1,a_2$ are Type-$3$ vertices. (b) $G[N_2]$. 
    The dotted line between~$V_1$ and~$V_2$ indicates that the two parts are non-adjacent.
    In accordance with \Cref{lem:TypeTriangles},~$b_1,b_2$ are not contained in a~$C_3$.
    $a_1,c_2,d_2$ forms a~$C_3$ as they are pairwise adjacent and none of the vertices are Type-$2$.}
    \label{fig:Example}
\end{figure}

We refer the reader to \Cref{fig:Example} to see an example of a~$C_4$-free graph of diameter-$2$ with~$\Delta=4$.
We show the degree of Type-$i$ vertices for~$i\in[3]$.

\begin{lem}\label{lem:T1Delta}
    A vertex of Type-$1$ is of degree~$\Delta$.
\end{lem}
\begin{proof}
    Let~$v\in V_i$ be of Type-$1$.
    Then~$v$ must cover all vertices in~$N_1$ via adjacency or via proxy. 
    Via adjacency, it can cover exactly one vertex~$x_i$, since by assumption,~$x_i$ has no neighbours in~$N_1$.
    All other~$\Delta-1$ vertices in~$N_1$ must be covered via proxy, and no two vertices in~$N_1$ can be covered by the same proxy, by \Cref{lem:N2SingleEdge}.
    So $\deg(v) = \Delta$.
\end{proof}

\begin{lem}\label{lem:T2Delta-1}
    A vertex of Type-$2$ is of degree~$\Delta-1$.
\end{lem}
\begin{proof}
    Let~$v\in V_i$ be of Type-$2$.
    Then~$v$ must cover all vertices in~$N_1$ via adjacency or via proxy. 
    Via adjacency, it can cover exactly one vertex~$x_i$; using~$x_i$ as proxy, it also covers another vertex~$x_j$.
    All other~$\Delta-2$ vertices in~$N_1$ must be covered via proxy, and no two vertices in~$N_1$ can be covered by the same proxy, by \Cref{lem:N2SingleEdge}.
    It cannot cover vertices in~$N_1$ twice via proxy, as that would form a~$C_4$.
    So $\deg(v) = \Delta-1$.
\end{proof}

\begin{lem}\label{lem:T3Delta}
	A vertex of Type-$3$ is of degree~$\Delta$.
\end{lem}
\begin{proof}
	This case is similar to the proof of \Cref{lem:T2Delta-1}.
	It requires the~$\Delta-1$ edges as described in the proof, and it has an additional edge to a neighbour in the same part.
	So the degree of the vertex must be~$\Delta$.
\end{proof}

So every vertex in~$N_2$ is of degree~$\Delta-1$ or~$\Delta$.
By definition, each part~$V_i$ can contain vertices of only Type-1, only Type-2, only Type-3, or only Type-2 and Type-3.
In the first three cases, we say that~$V_i$ is a \emph{part of Type-$j$}, for~$j\in[3]$.

The following lemma states that if~$N_2$ contains a Type-$2$ vertex, then we may assume that the canonical partition has a Type-$1$ part.
This will be helpful in \Cref{sec:ColorTypes} to reduce the number of cases.

\begin{lem}\label{lem:RootShift}
    If~$N_2$ contains a vertex of Type-$1$ or Type-$2$, then there exists a vertex~$w_0$ such that the canonical partition with~$w_0$ as root contains a Type-$1$ part.
\end{lem}
\begin{proof}
    We show an equivalent statement, that there exists a part which is adjacent to all other parts.
    Let~$v\in V_i$. 
    If~$v$ is of Type-$1$, then~$V_i$ is adjacent to all other parts by definition and we are done.
    So suppose~$v$ is of Type-$2$.
    By definition,~$x_i$ has a neighbour~$x_j$ in~$N_1$.
    So~$\deg(x_i) = |V_i| + 2$, where the~$+2$ term comes from the edges~$x_iv_0$ and~$x_ix_j$.
    As~$|V_i|= \Delta-2$ by \Cref{lem:adj=Delta-2}, we have~$\deg(x_i)=\Delta$.
    Take~$x_i$ as the root (instead of~$v_0$).
    Then~$v\in N_1(x_i)$, and~$v$ has no neighbours in~$N_1(x_i)$.
    In the corresponding canonical partition, 
    the part corresponding to~$v$ must be adjacent to all other parts, since~$v$ has no neighbours in~$V_i$.
    This gives the required statement.
\end{proof}

\subsection{Triangles in \texorpdfstring{$N_2$}{N2}}


The following lemma states that there are in general a lot of~$C_3$'s in~$G[N_2]$. 
These appear only between adjacent vertices which are not Type-$2$.
See \Cref{fig:graph_II}.

\begin{lem}\label{lem:TypeTriangles}
    Suppose~$|V_i| = \Delta-2$ for all~$i\in [\Delta]$.
    Let~$u,v\in N_2$ be adjacent vertices in different parts.
    Then~$uv$ is in a~$C_3$ in~$G[N_2]$ if and only if~$u$ and~$v$ are both not Type-$2$.
\end{lem}

\begin{proof}
    Let~$u\in V_i, v\in V_j$ such that~$uv\in E$.
    Let~$u,v$ be of Type-$a$ and Type-$b$, respectively, for~$a,b\in [3]$.
    We write~$(a,b)$ to consider the~$9$ possible cases (this reduces to~$6$ by symmetry). 
    If~$b\in\{2,3\}$, then let~$V_k$ denote the part which is non-adjacent to~$V_j$.
    If~$b=3$, then let~$w\in V_j$ denote the neighbour of~$v$ in~$V_j$.
    
    For the remainder of this proof, we write~$p(u)$ to denote the number of vertices in~$V_j$ which are covered by~$u$ via proxy.
    As~$|V_j|=\Delta-2$ by assumption and as~$u$ must cover all vertices of~$V_j$, we must have~$p(u)\in \{\Delta-3, \Delta-2\}$.
    Indeed,~$p(u)>\Delta-2$ will lead to a vertex being covered twice via proxy and hence a~$C_4$; on the other hand,~$p(u) < \Delta-3$ will mean that~$u$ does not cover all vertices of~$V_j$, violating the diameter-$2$ condition.
    If~$p(u) = \Delta-3$, then~$u$ covers each vertex of~$V_j$ exactly once.
    If~$p(u) = \Delta-2$, then one vertex must be covered twice, since~$v$ is covered via adjacency. 
    To prevent covering the same vertex twice via proxy, this would mean~$v$ must be covered once via adjacency and once via proxy. Then,~$u$ and~$v$ must lie in a~$C_3$.

    It remains to show that~$p(u) = \Delta-2$ if and only if~$(a,b)\in \{(1,1),(1,3),(3,3)\}$.
    
    \paragraph{Case 1: $(a,b)=(1,1)$.}
    By \Cref{lem:T1Delta}, since we do not consider the edge~$x_iu$ in~$G[N_2]$, we have $\deg_{G[N_2]}(u) = \Delta-1$.
    The edge~$uv$ does not cover any vertices of~$V_j$ via proxy.
    Each of the remaining~$\Delta-2$ edges will cover one vertex in~$V_j$ via proxy, since~$V_i$ and~$V_j$ are adjacent to all parts. So~$p(u) = \Delta-2$.
    \paragraph{Case 2: $(a,b)$ is $(1,2)$ (or~$(2,1)$).}
    By \Cref{lem:T1Delta}, $\deg_{G[N_2]}(u) = \Delta-1$.
    Again,~$uv$ does not contribute to~$p(u)$ just like in Case 1.
    Of the remaining~$\Delta-2$ edges, one is \emph{lost} (i.e. it does not cover anything in~$V_j$ via proxy) as~$V_k$ is adjacent to~$V_i$ and is non-adjacent to~$V_j$.
    Thus~$p(u) = \Delta-3$.
    In case of~$(2,1)$, we swap the roles of~$u$ and~$v$.

    \paragraph{Case 3: $(a,b)$ is~$(1,3)$ (or~$(3,1)$).}
    By \Cref{lem:T1Delta}, $\deg_{G[N_2]}(u) = \Delta-1$.
    Via~$uv$, we also cover~$w$ via proxy. 
    Of the remaining~$\Delta-2$ edges, one is lost as~$V_k$ is non-adjacent to~$V_j$.
    Thus,~$p(u) = 1 + \Delta-3 = \Delta-2$.
    
    \paragraph{Case 4: $(a,b) = (2,2)$.}
    By \Cref{lem:T2Delta-1}, $\deg_{G[N_2]}(u) = \Delta-2$.
    The edge~$uv$ does not contribute to~$p(u)$.
    From the remaining $\Delta-3$ edges, we get that~$p(u)\le \Delta-3$.
    However, we lose one edge as~$V_k$ is non-adjacent to~$V_j$.
    Thus,~$p(u) < \Delta-3$ which means this case is impossible.
    
    \paragraph{Case 5: $(a,b)$ is~$(2,3)$ (or~$(3,2)$).}
    By \Cref{lem:T2Delta-1}, $\deg_{G[N_2]}(u) = \Delta-2$.
    Via~$uv$, we also cover~$w$ via proxy.  
    Of the remaining~$\Delta-3$ edges, one is again lost to~$V_k$, and thus~$p(u) = 1 + \Delta-4 = \Delta -3$.
    (Compared to Case 4,~$u$ now also covers a vertex~$w$ of~$V_j$ via proxy using the edge~$uv$.)
    
    \paragraph{Case 6: $(a,b) = (3,3)$.}
    By \Cref{lem:T3Delta}, $\deg_{G[N_2]}(u) = \Delta-1$.
    Via~$uv$, we also cover~$w$ via proxy.
    Of the remaining~$\Delta-2$ edges, one is lost to~$V_k$, and thus~$p(u) = 1 + \Delta-3 = \Delta-2$.
\end{proof}

Note that if vertices~$u,v$ are adjacent and are in the same parts, then~$uv$ cannot be in a~$C_3$ in~$G[N_2]$.
If it was, then~$u,v$ must have a common neighbour in another part, which results in a~$C_4$.
The next corollary follows immediately from Case~$4$ in the proof of \Cref{lem:TypeTriangles}.

\begin{cor}\label{cor:AtMostOneType2}
	Two Type-$2$ vertices cannot be adjacent.
\end{cor}

\section{Necessary conditions for a proper \texorpdfstring{$3$}{3}-coloring}\label{sec:Color}

Throughout this section, we assume the existence of a proper 3-coloring~$c:V\rightarrow [3]$ and show some necessary conditions regarding the canonical partition and~$\Delta$.
Without loss of generality, we assume~$c(v_0) = 1$.
Then~$c(x_i)\in \{2,3\}$ for all~$i$, 
and so the vertices in~$V_i$ may only receive colors from~$[3]\setminus \{c(x_i)\} = \{1,5-c(x_i)\}$.
We say that~$V_i$ is of \emph{color class~$5-c(x_i)$}, denoted by~$c(V_i) = 5 - c(x_i)$.
For example, if~$c(x_i) = 2$, then~$V_i$ is of color class~$3$ or~$c(V_i) = 3$. 
If~$c(x_i) = 3$, then~$V_i$ is of color class~$2$ or~$c(V_i)=2$.

\begin{lem}\label{lem:NonAdjDiffColorClass}
	Suppose~$V_i$ and~$V_j$ are non-adjacent parts.
	If there exists a proper~$3$-coloring, then they must be of different color classes, i.e.,~$c(V_i)\ne c(V_j)$.
\end{lem}
\begin{proof}
	As~$x_i$ and~$x_j$ are adjacent,~$c(x_i)\ne c(x_j)$. So~$c(V_i)\ne c(V_j)$.
\end{proof}

Suppose~$c(V_1) = c(V_2) = 2$.
By the contrapositive of \Cref{lem:NonAdjDiffColorClass},~$V_1,V_2$ are adjacent.
Suppose also that both parts contain a non-Type-2 vertex~$u_1\in V_1, u_2\in V_2$ where~$u_1u_2\in E$.
By \Cref{lem:TypeTriangles},~$u_1$ and~$u_2$ lie in a~$C_3$. 
Since~$c(V_1) = c(V_2)$, the third point of the~$C_3$, say~$u_3$, must be in a part of color class~$3$.
We say that~$u_1u_2u_3$ is a~\emph{$(V_1,2)$-triangle}.
Formally, we call a~$C_3$ $xyz$ a \emph{$(V_i,2)$-triangle} if~$x\in V_i$ where~$c(V_i) = 2$, $y\in V_j$ where~$c(V_j) = 2$ and~$V_i\ne V_j$, and~$z\in V_k$ where~$c(V_k) = 3$.
In other words, a~$(V_i,2)$-triangle contains a vertex from~$V_i$, a vertex from the same color class as~$V_i$, and a vertex from a different color class from~$V_i$.

For example, consider \Cref{fig:graph_II}.
Suppose~$c(V_1) = c(V_3) = 2$ and~$c(V_2)=c(V_4) = 3$ (it is routine to check that this would not be extendable to a proper $3$-coloring).
Then~$a_1c_2d_2$ and~$a_2c_1d_1$ are two~$(V_1,2)$-triangles.
It is worth noting that they are also both~$(V_3,2)$-triangles.


\begin{lem}\label{lem:(V_1-2)TrianglesDisjoint}
	Suppose~$V_1$ contains a Type-$1$ vertex or a Type-$3$ vertex, and suppose~$c(V_1) = 2$.
	Consider two~$(V_1,2)$-triangles~$abc, xyz$ where~$a,x\in V_1$.
    The~$C_3$'s must be vertex disjoint except possibly at the vertex in~$V_1$, i.e.,~$\{b,c\}\cap\{y,z\}=\emptyset$.
\end{lem}
\begin{proof}
	Otherwise, we need to consider two cases.
    Assume without loss of generality that~$b=y$ if $\{b,c\}\cap\{y,z\}\ne\emptyset$.
    The two cases are $a\ne x, b=y$ and $a=x,b=y$.
    In the first case,~$b$ has two neighbours~$a,x$ in~$V_1$, which means~$x_1 a b x$ forms a~$C_4$. 
    In the second case,~$c,z$ are both neighbours of~$a$ and~$b$. So~$azbc$ forms a~$C_4$.
    In both cases we obtain a~$C_4$, which gives the required contradiction.
\end{proof}

\Cref{lem:(V_1-2)TrianglesDisjoint} ensures that vertices of color~$3$ in~$(V_1,2)$-triangles are distinct.

\begin{lem}\label{lem:Average}
	Suppose~$c(V_1) = 2$.
	Suppose there at least~$t$-many $(V_1,2)$-triangles.
	Let~$d$ denote the number of parts of color class~$3$.
	Then there exists a part of color class~$3$, which contains at least~$t/d$ vertices which must be given the color~$3$.
\end{lem}
\begin{proof}
	This follows immediately from~\Cref{lem:(V_1-2)TrianglesDisjoint}, as vertices with color~$3$ in~$(V_1,2)$-triangles must be distinct.
\end{proof}

\begin{lem}\label{lem:3PartsSameButNoColor}
	Suppose there are at least~$3$ parts~$V_i,V_j,V_k$ of the same color class.
	If there exists a proper $3$-coloring, then each part must contain the same even number~$\ell$ of vertices, and in each of $V_i,V_j,V_k$, exactly $\ell/2$ vertices must be colored~$1$ and exactly $\ell/2$ vertices must be colored~$c(V_i)$.
\end{lem}
\begin{proof}
	By the contrapositive of~\Cref{lem:NonAdjDiffColorClass},~$V_i,V_j,V_k$ are pairwise adjacent.
	By \Cref{lem:x_ix_jNonAdjacent}, there is a perfect matching on the vertices in any pair of parts~$V_i,V_j,V_k$.
	Suppose without loss of generality that~$c(V_i) = 2$.
	Suppose~$\alpha$ vertices of~$V_i$ are colored~$1$ and~$\ell - \alpha$ vertices are colored~$2$.
	Then~$\ell - \alpha$ vertices of~$V_j$ are colored~$1$ and~$\alpha$ are colored~$2$ (by the matching on~$V_i$ and~$V_j$), and similarly~$\ell - \alpha$ vertices of~$V_k$ are colored~$1$ and~$\alpha$ are colored~$2$ (by the matching on~$V_i$ and~$V_k$).
	By comparing the perfect matching on~$V_j$ and~$V_k$, we must have that~$\ell - \alpha$ vertices of~$V_k$ are colored~$2$.
	This means that~$\ell - \alpha = \alpha$, which gives~$\alpha = \ell/2$.
\end{proof}

\begin{cor}\label{cor:Deltage4Even}
	Suppose~$\Delta\ge4$. 
    If~$|V_i|=\Delta-2$ and if~$\Delta$ is odd,~$G$ is not $3$-colorable. 
\end{cor}
\begin{proof}
    Suppose there is a proper $3$-coloring.
	If~$\Delta\ge 5$, then there are at least three parts of the same color class.
	If~$|V_i|= \Delta-2$,
    then~$\Delta$ must be even by \Cref{lem:3PartsSameButNoColor}.
\end{proof}

\section{Conditioning on vertices of certain Types}\label{sec:ColorTypes}
In this section, we exhaustively check the 3-colorability of all possible graph configurations for~$\Delta\ge 17$.
We do this by showing non-$3$-colorability when~$N_2$ contains only Type-$1$ vertices (\Cref{lem:T1Delta-2} and \Cref{lem:T1Delta-1}), only Type-$3$ vertices~(\Cref{lem:ColorT3}), only Type-$1$ and Type-$2$ vertices (\Cref{lem:ColorT1andT2}), only Type-$1$ and Type-$3$ vertices (\Cref{lem:ColorT1T3}), and vertices of all three types (\Cref{lem:ColorT1T2T3}). 
We consider the cases based on proof similarity.
For all cases apart from \Cref{lem:ColorT1T2T3}, we show that the claims hold for~$\Delta\ge 4$.

We begin with the following lemma, which will be used in the proofs of \Cref{lem:T1Delta-2,lem:ColorT1andT2,lem:ColorT1T3}.

\begin{lem}\label{lem:T1TrianglesParity}
    Suppose~$\Delta\ge 4$, and
    let~$v\in N_2$ be a Type-$1$ vertex.
    Then the number of edges incident to~$v$ which are contained in a~$C_3$ is even. 
\end{lem}
\begin{proof}
    If the number in the claim is odd, then at least two of the~$C_3$'s must share an edge, which would result in a~$C_4$.
    This is a contradiction.
\end{proof}

\begin{lem}\label{lem:T1Delta-2}
	Suppose~$\Delta\ge 4$.
	Suppose~$N_2$ contains vertices only of Type-$1$, and suppose~$|V_1| = \Delta-2$.
    If such a graph exists, then it is not $3$-colorable. 
\end{lem}

\begin{proof}
    If~$\Delta$ is odd, then by~\Cref{cor:Deltage4Even}, the graph would not be $3$-colorable.
    So we assume~$\Delta$ is even.
    Let~$v\in V_1$.
    By \Cref{lem:TypeTriangles}, all incident edges of~$v$ in~$G[N_2]$ is in a~$C_3$.
    But there are exactly~$\deg_{G[N_2]}(v) = \Delta-1$ many such edges, which is an odd number of edges.
    This contradicts \Cref{lem:T1TrianglesParity}, and thus such graphs do not exist.
\end{proof}


\begin{lem}\label{lem:ColorT1andT2}
    Suppose~$\Delta\ge 4$.
    Suppose~$N_2$ contains vertices only of Type-$1$ and Type-$2$, and it contains at least one vertex of each type.
    If such a graph exists, then it is not~$3$-colorable.
\end{lem}
\begin{proof}
    By \Cref{lem:adj=Delta-2}, each part contains~$\Delta-2$ elements.
    By \Cref{cor:Deltage4Even}, the graph is not~$3$-colorable when~$\Delta$ is odd. So we assume~$\Delta$ is even.
    By \Cref{cor:AtMostOneType2}, there are~$\Delta-2$ Type-$1$ parts and~$2$ Type-$2$ parts.
    Relabelling as needed, let~$V_1$ be a Type-$1$ part, and let~$v\in V_1$.
    Observe that in~$G[N_2]$, exactly~$\Delta-3$ edges incident to~$v$ are contained in a~$C_3$, by \Cref{lem:T1Delta,lem:TypeTriangles}.
    Since~$\Delta$ is even, this number is odd.
    This contradicts \Cref{lem:T1TrianglesParity}, and thus such graphs do not exist.
\end{proof}

The proof of the following lemma follows a similar argument used in the proof of \Cref{lem:T1Delta-2}, and is therefore omitted.

\begin{lem}\label{lem:ColorT1T3}
	Suppose~$\Delta\ge 4$.
	Suppose~$N_2$ contains vertices only of Type-$1$ and Type-$3$, and it contains at least one vertex of each type.
	If such a graph~$G$ exists, then it is not $3$-colorable.
\end{lem}

\begin{lem}\label{lem:ColorT3}
    Suppose~$\Delta\ge 4$.
    Suppose~$N_2$ contains vertices only of Type-$3$. Then~$G$ is not $3$-colorable.
\end{lem}
\begin{proof}
    By \Cref{lem:adj=Delta-2}, $|V_i| = \Delta-2$.
	By \Cref{cor:Deltage4Even},~$\Delta$ is even.
    Relabelling as needed, let~$V_i,V_{i+1}$ for~$i=1,3,\ldots, \Delta-1$ be the non-adjacent pairs.
    This means that~$v_0x_ix_{i+1}$ forms a~$C_3$ for every~$i=1,3,\ldots, \Delta-1$.
    Non-adjacent parts belong to different color classes by \Cref{lem:NonAdjDiffColorClass}.
    By relabelling as needed, we assume
    \[c(V_i) = 
    \begin{cases}
    	2, \quad \text{if $i$ is even}\\
    	3, \quad \text{if $i$ is odd}\\
    \end{cases}
    \]
    Because the vertices in~$V_i$ form a perfect matching by definition of Type-$3$ vertices, half the vertices in~$V_i$ must be colored with~$1$ and the other half with~$c(V_i)$.
    
    By \Cref{lem:TypeTriangles} Case 6, all edges in~$G[N_2]$ are contained in a~$C_3$.
	There are exactly~$(\Delta/2 - 1)(\Delta-2)$ many $(V_2,2)$-triangles, each of which contains a unique vertex of color~$3$.
	As~$V_1$ and~$V_2$ are non-adjacent, there are~$\Delta/2 - 1$ parts of color class~$3$ which could contain such vertices.
    Then by \Cref{lem:Average}, there must be at least one part of color class~$3$, say~$V_j$, which has at least~$(\Delta/2 - 1)(\Delta-2) / (\Delta/2 - 1) = \Delta-2$ vertices with color~$3$.
    So all vertices of~$V_j$ must be of color~$3$.
    But this is not possible, since the vertices in~$V_j$ form a perfect matching.
\end{proof}

For the penultimate case, we need the definition of  a \emph{Moore graph}.
All \emph{$r$-regular} graphs (every vertex is of the same degree $d$) of diameter-$d$, have at most~$1+d\sum_{i=0}^{r-1} (d-1)^i$ vertices.
An~$r$-regular graph of diameter-$d$ that has exactly this many vertices is called a \emph{Moore graph}. 
Specifically for our case~$d=2$, a regular graph of diameter-$2$ and girth 5 is a Moore graph.
The degree of a diameter-$2$ Moore graph is known to be~$2,3,7,$ or $57$~\cite{hoffman1960moore}.
For~$2,3,7$, the corresponding diameter-$2$ Moore graphs are unique, and they are the cycle graph on~$5$ vertices, the Petersen graph, and the Hoffman-Singleton graph, respectively.
The existence of a diameter-$2$ Moore graph with degree-$57$ is still not known.

\begin{lem}\label{lem:T1Delta-1}
	Suppose~$\Delta\ge 4$.
	Suppose~$N_2$ contains vertices only of Type-$1$, and suppose~$|V_1| = \Delta-1$.
    Then~$G$ must be isomorphic to a Moore graph and it is not~$3$-colorable.
\end{lem}
\begin{proof}
	We show first that the graph~$G$ must be~$C_3$-free.
	Within~$G[N_2]$, each vertex is of degree~$\Delta-1$ by~\Cref{lem:T1Delta}.
	This implies that no vertex can cover another vertex in another part twice, as otherwise it would fail to cover all vertices in the part.
	So~$G[N_2]$ must be $C_3$-free.
	Observe that~$v_0$ and the vertices in~$N_1$ cannot be in a~$C_3$, since no two vertices of~$N_1$ nor two vertices in~$V_i$ are adjacent. 
	It follows that~$G$ must be~$C_3$-free, and thus~$G$ must be~$C_3$-free and~$C_4$-free.
    
    Observe that~$G$ must be regular. 
    Indeed, all vertices in~$N_1$ is of degree~$\Delta$ since each part has~$\Delta-1$ elements.
    Furthermore, all vertices in~$N_2$ are Type-$1$ by \Cref{lem:adj=Delta-2}, and thus of degree~$\Delta$ by \Cref{lem:T1Delta}.
    
    So~$G$ is a regular graph of diameter-$2$ and girth~$5$, and thus~$G$ must be isomorphic to a Moore graph of degree-$7$ or degree-$57$.
    The degree-$7$ variant, the Hoffman-Singleton graph, has chromatic number~$4$ (\cite{brouwer_hs}); using Hoffman's bound (see for example \cite{haemers2021hoffman}), the chromatic number of a Moore graph of degree-$57$ is at least $9$.
\end{proof}

For the final lemma, we require a stronger condition on the max degree~$\Delta$.

\begin{lem}\label{lem:ColorT1T2T3}
    Suppose~$\Delta\ge 17$.
    Suppose~$N_2$ contains vertices of all three types.
    Then~$G$ is not $3$-colorable.
\end{lem}
\begin{proof}
By \Cref{lem:adj=Delta-2},~$|V_i| = \Delta-2$.
By \Cref{cor:Deltage4Even},~$\Delta$ is even.
Let~$V_i$ be a part with the most Type-$2$ vertices, containing~$k$ Type-$2$ vertices and~$\Delta-2-k$ Type-$3$ vertices.
By assumption, there exists a part~$V_j$ which is non-adjacent to~$V_i$.
Then~$V_j$ contains at most~$k$ Type-$2$ vertices and at least~$\Delta-2-k$ Type-$3$ vertices.

The number of color classes in the graph is determined by the number of color classes of Type-$1$ parts.
By relabelling as needed, suppose there are at least~$\Delta/2$ parts of color class~$2$, and let~$V_1$ be a Type-$1$ part where~$c(V_1) = 2$.
We can count the number of~$(V_1,2)$-triangles in two ways.
We show by an average argument based on \Cref{lem:Average} that for certain~$k$ values, there must exist a part which contains at least~$\Delta/2$ vertices with color~$3$.
When there are at least 3 parts of color class~$3$, this gives the required non-3-colorability conclusion  by \Cref{lem:3PartsSameButNoColor}; we reach a similar conclusion when there are at most 2 parts of color class~$3$.

We count the number of~$(V_1,2)$-triangles in two ways.
Firstly, since Type-$2$ vertices cannot be adjacent by \Cref{cor:AtMostOneType2}, to accommodate for the~$k$ Type-$2$ vertices in~$V_i$, there must be at least~$k$ Type-$3$ vertices in every non-Type-$1$ part that is not~$V_i$ nor~$V_j$.
By \Cref{lem:TypeTriangles},
\begin{equation}\label{eqn:v12triangles1}
    \text{$\#$ of $(V_1,2)$-triangles}\ge (\Delta-2-k) + k(\Delta/2-2),
\end{equation}
where the first term comes from~$V_i$ or~$V_j$ depending on their color class, and the second term comes from having at least~$k$ Type-$3$ or Type-$1$ vertices in each of the remaining~$\Delta/2-2$ parts of color class~$2$ (the $-2$ term in~$\Delta/2-2$ comes from excluding~$V_1$, and one of $V_i$ or $V_j$).
Note that the Type-$1$ parts actually contribute~$\Delta-2$ many~$(V_1,2)$-triangles, but for our purposes,~$k$ suffices.

Secondly, each non-Type-$1$ part that is not~$V_i$ nor~$V_j$ contains at least~$\Delta-2-k$ Type-$3$ vertices by choice of~$k$. 
Each Type-$1$ part contains~$\Delta-2$ Type-$1$ vertices, and therefore at least~$\Delta-2-k$ Type-$1$ vertices. By \Cref{lem:TypeTriangles},
\begin{equation}\label{eqn:v12triangles2}
    \text{$\#$ of $(V_1,2)$-triangles}\ge (\Delta-2-k)(\Delta/2-1).
\end{equation}
We split into three cases based on the value of~$k$.

\paragraph{Case 1: $k\ge \Delta/2 + 2$.}
By \Cref{lem:Average} and 
\Cref{eqn:v12triangles1}, there is a part of color class~$3$ which contains at least
\[\displaystyle\frac{(\Delta-2-k) + k(\Delta/2-2)}{\Delta/2} \ge \frac{(\Delta/2+2)(\Delta/2-2)}{\Delta/2} = \frac{\Delta}{2} - \frac{8}{\Delta}\]
vertices with color~$3$.
As this quantity is an integer and since~$\Delta\ge 17$, we conclude that there are at least~$\Delta/2$ vertices with color~$3$ in one part.

Since each part contains~$\Delta-2$ elements, we have the required conclusion if there are at least three parts of color class~$3$ by \Cref{lem:3PartsSameButNoColor}.
If there are at most two parts of color class~$3$, then there must be at least (using \Cref{eqn:v12triangles1} again)
\[(\Delta-2-k) + k(\Delta/2-2) \ge (\Delta/2+2)(\Delta/2-2) = \Delta^2/4 -4\]
vertices with color~$3$. But there can be at most~$2\Delta-4$ vertices with color~$3$.
Since~$\Delta\ge 17$, we have~$\Delta^2/4 -4> 2\Delta-4$.
So this case is also not~$3$-colorable.

\paragraph{Case 2: $k\le \Delta/2 -3$.}
By \Cref{lem:Average} and \Cref{eqn:v12triangles2}, there is a part of color class~$3$ which contains at least
\[\displaystyle\frac{(\Delta-2-k)(\Delta/2-1)}{\Delta/2} \ge \frac{(\Delta/2+1)(\Delta/2-1)}{\Delta/2} = \frac{\Delta}{2} - \frac{2}{\Delta}\]
vertices with color~$3$.
As this quantity is an integer and since~$\Delta\ge 17$, we conclude that there are at least~$\Delta/2$ vertices with color~$3$ in one part.
Since each part contains~$\Delta-2$ elements, we have the required conclusion if there are at least three parts of color class~$3$ by \Cref{lem:3PartsSameButNoColor}.
If there are at most two parts of color class~$3$, then there must be at least (using \Cref{eqn:v12triangles2} again)
\[(\Delta-2-k)(\Delta/2-1) \ge (\Delta/2+1)(\Delta/2-1) = \Delta^2/4 -1\]
vertices with color~$3$. But there can be at most~$2\Delta-4$ vertices with color~$3$.
Since~$\Delta\ge 17$, we have~$\Delta^2/4 -1> 2\Delta-4$.
So this case is also not~$3$-colorable.

\paragraph{Case 3: $\Delta/2-2 \le k \le \Delta/2+1$.}
Noting that~$\deg(x_i) = \Delta$, we consider the canonical partition of~$N_2$ rooted at~$x_i$.
$V_i$ contains~$k$ Type-$2$ vertices; each of these vertices has a Type-$1$ part corresponding to it in the new partition.
In this new partition, we shall reuse the notation~$V_1,V_2,\ldots, V_\Delta$. 
By relabelling as needed, suppose there are at least~$\Delta/2$ types of color class~$2$, and let~$V_1$ be a Type-$1$ part where~$c(V_1) = 2$.

We count the number of~$(V_1,2)$-triangles.
There are exactly~$k$ Type-$1$ parts, which means that at least $k/2$ of these parts are of color class~$2$. 
From these parts, we pick up a contribution of~$(k/2-1)(\Delta-2)$.
The $-1$ term in the~$k/2-1$ expression comes from excluding~$V_1$.

On the other hand, there are exactly~$\Delta-k$ non-Type-$1$ parts.
Let~$V_k$ be a non-Type-$1$ part with the most number~$j$ of Type-$2$ vertices.
Let~$V_\ell$ be its non-adjacent part.
If~$j\le\Delta/2$, then all non-Type-$1$ parts contain at least~$\Delta/2$ Type-$3$ vertices by maximality of~$j$.
On the other hand, if~$j\ge \Delta/2$, then all non-Type-$1$ parts which are adjacent to~$V_k$ must have at least~$j$ Type-$3$ vertices by \Cref{cor:AtMostOneType2}.
This means that on average, each non-Type-$1$ part of color class~$2$ has at least~$\Delta/2-1$ Type-$3$ vertices. 
Indeed, one can match the~$\Delta-2-j$ Type-$3$ vertices of~$V_k$ (or of~$V_\ell)$ with the~$j$ Type-$3$ vertices of an adjacent part of color class~$2$.
Such a match exists since~$\Delta-k\ge \Delta/2-1$, which means there are at least 8 non-Type-$1$ parts as~$\Delta\ge 17$.
To conclude, we obtain a contribution of at least~$((\Delta-k)/2)(\Delta/2-1)$ many~$(V_1,2)$-triangles.

By \Cref{lem:Average}, there is a part of color class~$3$ with at least
\[\frac{(k/2-1)(\Delta-2) + ((\Delta-k)/2)(\Delta/2-1)}{\Delta/2}\]
vertices with color~$3$.
It is routine to show that this number is at least~$\Delta/2$ for $\Delta\ge 17,15,13,11$, respectively, when~$k=\Delta/2-2, \Delta/2-1, \Delta/2, \Delta/2 + 1$.
Since~$\Delta\ge 17$, this means that the graphs in this case are not~$3$-colorable.
\end{proof}

We summarize the results obtained in this section.

\begin{thm}\label{thm:Delta17}
    Let~$G$ be a~$C_4$-free graph of diameter-$2$ with no universal vertex, where~$\Delta(G)\ge17$.
    Then~$G$ is not~$3$-colorable.
\end{thm}
\begin{proof}
    Let~$v_0$ be a vertex with~$\deg(v_0) = \Delta = \Delta(G)$, and let~$N_2 = N_2(v_0)$.
    Since there is no universal vertex,~$N_2\ne \emptyset$.
    Every vertex in~$N_2$ is of Type-$1$, Type-$2$, or of Type-$3$ by \Cref{lem:1of3Types}.
	By \Cref{lem:RootShift}, if there is either a Type-$1$ or a Type-$2$ vertex, then we may always assume that there is a Type-$1$ part in the canonical partition.
	Furthermore, every part is of the same cardinality (\Cref{lem:V_iUniformSize}) and every part has~$\Delta-2$ or~$\Delta-1$ elements (\Cref{lem:V_iSizeLowerBound}).
	We have exhaustively checked that when each part contains~$\Delta-2$ elements, the graph is not~$3$-colorable (\Cref{lem:T1Delta-2,lem:ColorT1andT2,lem:ColorT3,lem:ColorT1T3,lem:ColorT1T2T3}).
	On the other hand, if each part contains~$\Delta-1$ elements, then~$G$ must contain only Type-$1$ parts (\Cref{lem:adj=Delta-2}).
	In that case,~$G$ must be isomorphic to a Moore graph of $57$, which is not~$3$-colorable (\Cref{lem:T1Delta-1}).
    This proves the claim.
\end{proof}

\section{Proof of \texorpdfstring{\Cref{thm:main}}{Theorem 1.1}}\label{sec:Proof}

We are now ready to prove \Cref{thm:main}.

\LTCDtwfour*
\begin{proof}
    Let~$G$ be a~$C_4$-free graph of diameter-$2$.
    If~$N_2=\emptyset$, then we can solve the instance in polynomial-time by \Cref{lem:N2Empty}.
    So suppose~$N_2\ne \emptyset$.
    By \Cref{thm:Delta17}, all graphs with maximum degree at least~$17$, are not~$3$-colorable. 
    The list $3$-colorability of the remaining finitely many graphs with maximum degree at most~$16$ can be hardcoded into a lookup table. 
	One can determine the maximum degree in polynomial-time, and thus deciding if~$G$ is list~$3$-colorable can be solved in polynomial-time.
\end{proof}

\section{Discussion}\label{sec:Discussion}

We showed that \LTCDCF is polynomial-time solvable.
It turns out that many~$C_4$-free graphs of diameter-$2$ with maximum degree at least~$17$ are not $3$-colorable (\Cref{thm:Delta17}).

The key ingredient was the rigid structure arising from the $C_4$-free and diameter-$2$ conditions.
By anchoring on a vertex~$v_0$ of maximum degree~$\Delta$, one can show that every vertex in the graph must be of degree~$\Delta$ or~$\Delta-1$.
Furthermore, one can partition the vertices which are distance-$2$ away from~$v_0$ into~$\Delta$ parts.
Each part contains the same number of vertices and if they are adjacent, the vertices form a perfect matching.

The $\Delta\ge 17$ condition for non-$3$-colorability arose from ensuring an impossible number of vertices of a certain color in the proof of \Cref{lem:ColorT1T2T3}. 
Interestingly, it was only needed in one part of the proof; the other parts required weaker conditions on~$\Delta$. 
Furthermore, non-$3$-colorability was shown for~$\Delta\ge 4$ in all other results in the same section (\Cref{lem:T1Delta-2,lem:ColorT1andT2,lem:ColorT3,lem:ColorT1T3}).
So a natural question is to ask if we can strengthen the~$\Delta\ge 17$ result to~$\Delta\ge 4$.
An affirmative answer would imply that the only~$C_4$-free graphs of diameter-$2$ with~$N_2\ne \emptyset$ are the three leftmost graphs in \Cref{fig:PotentialYESinstances}.

\begin{prob}
    Let~$G$ be a~$C_4$-free graph of diameter-$2$ with~$\Delta(G)\ge4$. 
    Show that~$G$ is not~$3$-colorable.
\end{prob} 

In another direction, one can ask if the so-called \emph{diamond elimination} of \cite{mertzios2016algorithms} preserves induced~$C_4$-freeness. 

\begin{rrule}[Diamond Elimination]\label{rr:Diamond}
    If the quadruple~$\{u_1,u_2,u_3,u_4\}$ of vertices in~$G$ induces a diamond graph, where~$u_1u_2\notin E$, then merge vertices~$u_1$ and~$u_2$.
\end{rrule}
\noindent This is not immediately clear 
(see \Cref{fig:DiamondElimination}).
An affirmative answer would allow one to strengthen the~`$C_4$-free' condition to the `induced~$C_4$-free' condition.
It should be noted that the graph shown in \Cref{fig:DiamondElimination} is of diameter-$3$.

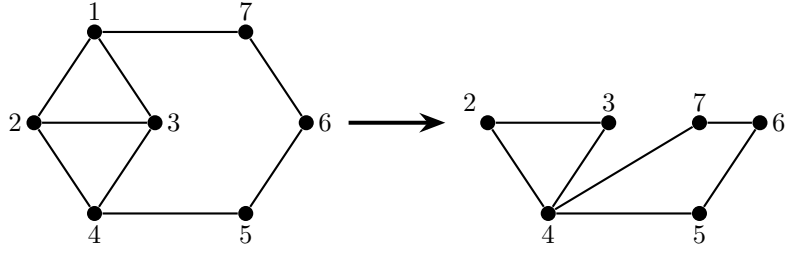
\begin{figure}
    \centering
    \begin{tikzpicture}[
    scale=0.8,
    vertex/.style={circle, fill=black, inner sep=2pt},
    label distance=-2pt
    ]
    
    \begin{scope}[shift={(0,0)}]
        \node[vertex] (v1) at (2,4) [label=above:$1$] {};
        \node[vertex] (v2) at (1,2.5) [label=left:$2$] {};
        \node[vertex] (v3) at (3,2.5) [label=right:$3$] {};
        \node[vertex] (v4) at (2,1) [label=below:$4$] {};
        \node[vertex] (v5) at (4.5,1) [label=below:$5$] {};
        \node[vertex] (v6) at (5.5,2.5) [label=right:$6$] {};
        \node[vertex] (v7) at (4.5,4) [label=above:$7$] {};
        
        \draw[thick] (v1) -- (v2);
        \draw[thick] (v1) -- (v3);
        \draw[thick] (v1) -- (v7);
        
        \draw[thick] (v2) -- (v3);
        \draw[thick] (v2) -- (v4);
        
        \draw[thick] (v3) -- (v4);
        
        \draw[thick] (v4) -- (v5);
        \draw[thick] (v5) -- (v6);
        \draw[thick] (v6) -- (v7);
    \end{scope}
    
    \draw[-{Stealth[scale=1]}, line width=1.5pt] (6.2, 2.5) -- (7.8, 2.5);
    
    \begin{scope}[shift={(7.5,0)}]
        \node[vertex] (v2) at (1,2.5) [label=above left:$2$] {};
        \node[vertex] (v3) at (3,2.5) [label=above:$3$] {};
        \node[vertex] (v4) at (2,1) [label=below:$4$] {};
        \node[vertex] (v5) at (4.5,1) [label=below:$5$] {};
        \node[vertex] (v6) at (5.5,2.5) [label=right:$6$] {};
        \node[vertex] (v7) at (4.5,2.5) [label=above:$7$] {};
        
        \draw[thick] (v2) -- (v3);
        \draw[thick] (v2) -- (v4);
        \draw[thick] (v3) -- (v4);
        \draw[thick] (v4) -- (v7);
        \draw[thick] (v4) -- (v5);
        \draw[thick] (v7) -- (v6);
        \draw[thick] (v6) -- (v5);
    \end{scope}
    
    \end{tikzpicture}
    \caption{Diamond elimination (\Cref{rr:Diamond}) does not preserve induced~$C_4$-freeness.
    We merge the vertices~$1$ and~$4$ in the left graph to obtain the right graph.
    The right graph contains an induced~$C_4$.
    Observe that the left graph is of diameter-$3$.}
    \label{fig:DiamondElimination}
\end{figure}

\begin{prob}
    Let~$G$ be an induced~$C_4$-free graph of diameter-$2$. Let~$G'$ be the graph obtained by applying \Cref{rr:Diamond} exhaustively to~$G$.
    Is~$G'$ induced~$C_4$-free?
\end{prob}

Finally, the problem remains open when we replace the~$C_4$-free condition by the~$C_3$-free condition, or if we remove the $C_4$-free condition outright.
A natural question is to ask if any of the structural results of \Cref{sec:Characterization} remain true.
The answer is a definitive no. \Cref{lem:N2SingleEdge}, which states that every vertex in~$N_2$ has exactly one neighbour in~$N_1$ already breaks, as the `exactly' becomes `at least'. 
So vertices in~$N_2$ cannot be partitioned based on their adjacencies to~$N_1$, which means that the canonical partition would no longer exist.
A new strategy is most likely needed to tackle these problems.

\section*{Acknowledgements}
The author thanks Carla Groenland, Ananthakrishnan Ravi, and Nikolaas Verhulst for helpful discussions on $3$-coloring graphs of bounded degree, Dion Gijswijt for a helpful discussion on diamond elimination, and Gabriëlle Zwaneveld for suggesting the use of Hoffman's bound. 

\bibliographystyle{plain}
\bibliography{z_bib}

\end{document}